\documentclass[graybox]{svmult}
\usepackage{mathptmx}       
\usepackage{helvet}         
\usepackage{courier}        
\usepackage{makeidx}         
\usepackage{graphicx}        
\usepackage{multicol}        
\usepackage[bottom]{footmisc}
\usepackage{amsmath,amssymb,amsfonts}        
\makeindex             

\usepackage{enumerate}
\usepackage{float}
\usepackage{subfig}

\usepackage{bm} 

\newcommand{\half}{\mbox{$\frac12$}}
\renewcommand{\Re}{\operatorname{Re}}
\renewcommand{\Im}{\operatorname{Im}}

\DeclareMathOperator{\var}{\normalfont{\text{Var}}}

\newcommand{\dto}{\Rightarrow}

\newcommand{\Xbb}{\widetilde{X}_{n}}
\newcommand{\Xcc}{\widetilde{\mathbf{X}}_{n}}

\newcommand{\scale}{w}

\newcommand{\driftX}{\mu} 
\newcommand{\driftGauss}{\mu}

\newcommand{\varianceGauss}{\sigma^2}

\newcommand{\Dn}{\Delta_n}

\newcommand{\pbb}{\pi}

\newcommand{\Cf}{C}

\title*{On extension of the Markov chain approximation method for computing Feynman--Kac type expectations}
\author{Vincent Liang and Konstantin Borovkov}
\institute{Vincent Liang \at School of Mathematics and Statistics, The University of Melbourne, Parkville 3010, Australia \email{vliang@student.unimelb.edu.au}
\and Konstantin Borovkov \at School of Mathematics and Statistics, The University of Melbourne, Parkville 3010, Australia \email{borovkov@unimelb.edu.au}}

\begin{document}
\maketitle

\abstract{
An efficient discrete time and space Markov chain approximation employing a Brownian bridge correction for computing curvilinear boundary crossing probabilities for general diffusion processes was recently proposed in Liang and Borovkov (2021). One of the advantages of that method over alternative approaches is that it can be readily extended to computing expectations of path-dependent functionals over the  event of the process trajectory staying between two curvilinear  boundaries. In the present paper, we extend the scheme to compute expectations of the Feynman--Kac type that frequently appear in option pricing. To illustrate our approximation scheme, we apply it in three special cases. For sufficiently smooth integrands, numerical experiments suggest that the proposed approximation converges at the rate $O(n^{-2})$, where $n$ is the number of steps on the uniform time grid used.
}

\section{Introduction}

Barrier options are widely used path-dependent derivative securities, and there is extensive literature on their valuation (see e.g.~\cite{Armstrong2001}, \cite{Carr1995} and the references therein). Most of the literature is focused on the case where the barriers are ``flat''. However, curved boundaries also appear in the context of pricing options on underlying assets that have a volatility term structure. Using a deterministic time-change, the time dependence in the volatility can be transferred into the boundary, reducing the problem of pricing call and put options in the Black--Scholes type setting with deterministic interest rates to computing the curvilinear boundary crossing probability of a diffusion process (see e.g.~Section 3 in \cite{Borovkov2005}). Methods for solving the latter problem are numerous, and we refer the interested reader to \cite{Liang2021} for a literature review. To extend the aforementioned problem to the case of stochastic interest rates, one needs to introduce a numeraire process.

It is well-known that, under the no-arbitrage assumption, the fair price of a replicable derivative on an underlying asset $X$ with maturity $T$ and general payoff $\phi(X(T))$ is equal to $\mathbf{E}\phi(X(T))/N(T)$, where $\mathbf{E}$ denotes the expectation with respect to the martingale measure associated with the numeraire process $N$ (for details, see e.g. Chapter 6 in \cite{Bingham2004}). If we introduce a barrier feature for the above option using two time-dependent barriers $g^-<g^+$, the payoff of the new option will be given by
	\begin{equation*}
		\phi(X(T))\mathbf{1}\{g^-(s)<X(s) <g^+(s),\,s \in [0,T]\}.
	\end{equation*}
Hence the fair price of the barrier option is equal to
\begin{equation}\label{expect}
	\mathbf{E}\bigl[\phi(X(T))/N(T);g^-(s)< X(s) <g^+(s),\,s \in [0,T]\bigr].
\end{equation}

The purpose of this paper is to extend the Markov chain approximation method proposed in~\cite{Liang2021} to compute the above expectation in the case when the numeraire process takes the form $N(T) = \exp\{\int_0^T V(X(t))\,dt\}$ for some continuous function~$V$, i.e.~to evaluate
\begin{equation}\label{eqn:feynmann_kac}
	\mathbf{E}\Bigl[e^{-\int_0^TV(X(u))\,du}\phi(X(T));g^-(s) < X(s) <g^+(s),\,s \in [0,T]\Bigr].	
\end{equation}
The method suggested in~\cite{Liang2021} was a development of the approach from~\cite{FuWu2010} for computing boundary non-crossing probabilities of the Brownian motion process. That approach suggested to replace the continuous dynamics of the underlying Brownian motion process  with those of a denumerable    discrete time Markov chain, by discretising both time and space using uniform grids. The transition probabilities for these Markov chains were specified to be proportional to the values of the Brownian motion transition densities for the respective time and space increments. The desired approximation for the boundary non-crossing probability was then computed  in~\cite{FuWu2010} by multiplying finite-dimensional substochastic matrices obtained by retaining the entries corresponding to the space nodes located between the given boundaries at the respective times. 

Our modification of this method developed in~\cite{Liang2021} dramatically improved the convergence rate of the procedure by shifting and adjusting the uniform space grids so that they would have nodes lying exactly on the boundaries, and by implementing one-step Brownian bridge corrections taking into account the possibility that the process could cross a boundary at a time between the points on the discrete time grid. At the same time, the method was extended in~\cite{Liang2021}  to a much broader class of diffusions.

Roughly speaking, one can think about the workings of this method as summing up  the probabilities of all the Markov chain trajectories lying between the boundaries (and applying appropriate corrections). As, for each trajectory, its probability is given by the product of one-step transition probabilities, one basically ``tracks" each trajectory of the chain  in this computational process. Unlike alternative approaches to computing boundary crossing probabilities, such as the method of integral equations (see e.g.~\cite{GuRiRoTo1997} and further references therein), this provides one with an opportunity to compute not only probabilities but also expectations of path-dependent functionals, finding the values of expressions of the form~\eqref{eqn:feynmann_kac} being of particular convenience.

Turning back to applications related to~\eqref{expect}, we note that when the underlying~$X$ is the spot interest rate, this includes the case when the numeraire is the bank account process, i.e. $V(x) = x$. This expression can also be used to price what we call a \textit{hybrid step-barrier option}, an option that has a step option feature (see \cite{Linetsky1999}) and a barrier above the step level (see Example~3 below).

It was reported in~\cite{Liang2021} that numerical experiments had showed that the proposed  method  produces approximations for boundary crossing probabilities with convergence rate $O(n^{-2})$, where~$n$ is the number of nodes on the regular time discretisation grid. It turns out that this convergence rate is preserved in the new extension of the method presented in this paper when~$V$ is sufficiently smooth. Note that proving the above-mentioned very fast convergence rate for our scheme is a very difficult task already in the basic case of computing boundary crossing probabilities. Establishing this rate in the one-sided boundary case with $g^+\in C^2$ for the Brownian motion process is work in progress.

The paper is organised as follows. In Section~\ref{section:method} we describe our approximation procedure and demonstrate its convergence.  In Section~\ref{section:numerical} we present the results of applying it in three special cases.

\section{The method}\label{section:method}

We will begin by specifying the diffusion process model for the underlying asset~$X$. We assume a unit diffusion coefficient since we can transform a large class of diffusion processes to this case using the well-known unit-diffusion transform (see e.g.~Section 2 in \cite{Liang2021}). Then, we will describe the sequence of Markov chain approximations and prove that the sequence converges weakly to the desired diffusion process. Using the weak convergence result, we will show that the corresponding approximations to the option price that are using the Brownian bridge correction techniques converge as well. Recall that employing this correction dramatically improves convergence rates in the basic problem of computing boundary crossing probabilities \cite{Liang2021}. As we will see this applies in the present paper setup as well.

Without loss of generality, we set the terminal time $T := 1$ since a deterministic time change can be applied to get the results for a general fixed $T <\infty$.

\subsection{The setup}

Suppose that the underlying asset's price is modelled by the one-dimensional diffusion process
    \begin{equation*}
        X(t) = x_0 + \int_0^t \driftX(s,X(s))\,ds +W(t),\quad t\geq0,
    \end{equation*}
where $\{W(t)\}_{t\geq 0}$ is a standard Brownian motion process, and $x_0$ is non-random. Assume the following condition is satisfied:
\begin{enumerate}
    \item [(C)] For any fixed $x \in \mathbb{R}$, one has $\driftX(\,\cdot\,,x) \in C^1([0,1])$, and for any fixed $t \in [0,1]$, one has $\driftX(t,\,\cdot\,) \in C^2(\mathbb{R})$. Moreover, for any $r>0$, there exists a $K_r <\infty$ such that one has
        \begin{equation*}
            \lvert \driftX(t,x)\rvert + \lvert \partial_t \driftX(t,x)\rvert + \lvert \partial_{x} \driftX(t,x) \rvert + \lvert \partial_{xx} \driftX(t,x) \rvert \leq K_r , \quad t\in[0,1], \,\lvert x \rvert \leq r.
        \end{equation*}
\end{enumerate}
We will also need some notations and conditions related to the boundaries $g^{\pm}$. Denote by $\Cf=C([0,1])$ the space of continuous functions $f : [0,1]\to\mathbb{R}$ equipped with the uniform norm $\| f\|_{\infty} := \sup_{t\in[0,1]}\lvert f(t)\rvert$. For a fixed $x_0 \in \mathbb{R}$, consider the class
    \begin{equation*}
        \mathcal{G}:= \Bigl\{ (f^-,f^+): f^{\pm}\in \Cf,\,f^-(0) < x_0 < f^+(0),\, \min_{0\leq t \leq 1}(f^+(t)-f^-(t)) > 0 \Bigr\}
    \end{equation*}
of pairs of functions from $\Cf$ and introduce the notation
    \begin{equation*}
        S(f^-,f^+):= \{v \in \Cf:f^-(t) < v(t) <f^+(t),\,t\in [0,1] \},\quad (f^-,f^+) \in \mathcal{G} 
    \end{equation*}
for the  ``bunch'' of continuous functions whose graphs are entirely contained in the strip between the boundaries~$f^\pm.$
    
Let $V : \mathbb{R} \to \mathbb{C}$ be a (possibly complex-valued) continuous function and $\phi: \mathbb{R}\to \mathbb{R}$ be some continuous function. The problem we deal with in this paper is how to compute
    \begin{equation}\label{defn:P}
		Q := \mathbf{E}\Bigl( e^{-\int_0^1V(X(s))\,ds} \phi(X(1)) ; X \in G\Bigr),\quad G:= S(g^-,g^+),
    \end{equation}
for some $(g^-,g^+) \in \mathcal{G}$. To explain our approach, set  $\mathbf{E}^{s,x} (\cdot):=\mathbf{E} (\cdot \,|\,X(s)=x) $ and introduce Feynman--Kac type time-dependent semigroup $\{T_{s,t}\}_{0\leq s< t \leq 1}$ defined by
\begin{equation}\label{T}
	T_{s,t}f(x) = \mathbf{E}^{s,x}\Bigl( e^{-\int_s^t V(X(u))\,du}f(X(t)) ; g^{-}(u) < X(u) < g^+(u),u\in[s,t]\Bigr)
\end{equation}
and represent \eqref{defn:P} as
\begin{equation}\label{Q}
	Q = (T_{0,t_{n,1}}T_{t_{n,2},t_{n,3}}\cdots T_{t_{n,n-1},t_{n,n}}\phi)(x_0),
\end{equation}
for
\begin{equation*}
	t_{n,k}:=k/n,\quad k=0,1,\ldots,n,
\end{equation*}
the uniform partition of $[0,1]$ of rank $\Delta_n:=1/n$, $n\geq 1$. The idea of our method is to approximate operators~$T_{s,t}$ with their discrete versions. We will implement it in the next section, leading to approximation \eqref{defn:P_n} to our~$Q$.

\subsection{Markov chain approximation}

To specify our time-dependent Markov chain approximation, we first define the space grids $E_{n,k}$ and the transition probabilities $p_{n,k}$, and then introduce the corrective terms $\pi_{n,k}$ and $e_{n,k}$ which account for boundary correction and the presence of the term with $V$ respectively.

The grids $E_{n,k}$ are constructed as follows. Set $g_{n,k}^{\pm} := g^{\pm}(t_{n,k})$, $k=1,\ldots,n,$ and, for fixed $\delta\in (0,\half]$ and $\gamma>0$, put
    \begin{equation*}
        \scale_{n,k}   :=
            \begin{cases}
                \frac{(g_{n,k}^+ -g_{n,k}^-)/\Dn ^{1/2 +\delta} }{\lfloor\gamma(g_{n,k}^+ -g_{n,k}^-)/\Dn ^{1/2 +\delta} \rfloor },& 1 \leq k<n,\\
                \frac{(g^+(1) -g^-(1))/\Dn   }{\lfloor\gamma(g^+(1) -g^-(1))/\Dn   \rfloor},& k=n,
            \end{cases}
    \end{equation*}
assuming that~$n$ is large enough such that the integer parts in all the denominators are positive. We set the time-dependent space lattice step sizes to be
    \begin{equation*}\label{eqn:h}
        h_{n,k} :=
        \begin{cases}
            \scale_{n,k}\Dn ^{1/2 +\delta},& 1 \leq k < n, \\
            \scale_{n,n}\Dn  , & k =n.
        \end{cases}
    \end{equation*}
Next, we define
    \begin{equation*}
    \label{grid_E}
        E_{n,k}:= \{ g_{n,k}^+ - j h_{n,k} : j \in \mathbb{Z}\},\quad k = 1,\ldots,n.
    \end{equation*}
We also put $E_{n,0}:=\{x_0\}$ and define the corresponding boundary restricted lattices
    \begin{equation*}
        E_{n,k}^G:= \{ x \in E_{n,k}: g_{n,k}^- < x <g_{n,k}^+ \},\quad k = 1,\ldots,n.
    \end{equation*}
Further, for $k=1,\ldots,n,$ we introduce the discrete time drift and diffusion coefficients
\begin{equation}\label{eqn:diffusion_coef}
    \begin{aligned}
        \driftGauss_{n,k}(x) &:= \big(\driftX + \half\Dn  (\partial_t \driftX+\driftX\partial_x\driftX + \half\partial_{xx}\driftX)\big)(t_{n,k-1},x)\Dn,\\
        \varianceGauss_{n,k}(x) &:= ( 1 + \half\Dn   \partial_x \driftX  (t_{n,k-1},x))^2\Dn,
    \end{aligned}
   \end{equation}
and define one-step ``pseudo-transition probabilities'' (we use this expression because they do not sum to one, although the normalising constants converge to~1 so quickly as~$n\to\infty$ that it makes no sense to normalise them, see Remark 7 in \cite{Liang2021}) as
    \begin{equation*}
       p_{n,k}(x,y) :=
        \varphi(y \,|\,  x + \driftGauss_{n,k}(x),\varianceGauss_{n,k}(x))h_{n,k},\quad (x,y) \in E_{n,k-1}\times E_{n,k},
    \end{equation*}
where $\varphi(x\,|\,\mu,\sigma^2):= (2\pi \sigma^2)^{-1/2}e^{-(x -\mu)^2/(2\sigma^2) }$, $x,\mu\in\mathbb R$, $\sigma^2>0$. We will implement the Brownian bridge correction for two-sided boundary crossing probabilities by using the factor
 \begin{equation*}
        \pbb_{n,k}(x,y) =  1 - \mbox{$\exp\{ \frac{-2}{\Dn  }(g_{n,k-1}^+ - x)(g_{n,k}^+ -y) \}$}- \mbox{$\exp\{ \frac{-2}{\Dn }( g_{n,k-1}^- - x)(g_{n,k}^- -y) \}$}.
    \end{equation*}
Note that we ignore here the highly unlikely event that the trajectory of the Brownian bridge process hits both the upper and lower boundaries in the small time interval $[t_{n,k-1},t_{n,k}]$, see \cite{Liang2021}. The following expression will be used as a trapezoidal approximation of the exponential term $\exp\{-\int_{(k-1)/n}^{k/n}V(X(s))\,ds \}$:
    \begin{equation}\label{defn:potential_correction}
        e_{n,k}(x,y):= \exp\{\mbox{$\frac{-\Dn}{2}$}(V(x) +V(y) )\}.
    \end{equation}
Now we introduce the pseudo-transition matrices $S_{n,k} \in \mathbb{R}^{\lvert E_{n,k-1}^G\rvert\times \lvert E_{n,k}^G\rvert}$ involving both the boundary crossing correction terms $\pi_{n,k}$ and our adjustments $e_{n,k}$ for the presence of the exponential terms with $V$:
\begin{equation*}
    S_{n,k} := \big[p_{n,k}(x,y)\pi_{n,k}(x,y)e_{n,k}(x,y)\big]_{(x,y)\in E_{n,k-1}^G\times E_{n,k}^G}.
\end{equation*}
Denoting by $\phi_n := [\phi(x)]_{x \in E_{n,n}^G}$ the $n$-th payoff column vector, our approximation to~\eqref{defn:P} is given by the sequence of matrix products
\begin{equation}\label{defn:P_n}
    Q_{n} := S_{n,1}S_{n,2}\cdots  S_{n,n}\phi_n,\quad n \geq 1,
\end{equation}
which are discrete analogues of~\eqref{Q}. 

\subsection{Weak convergence}

\begin{theorem}
Let condition \emph{(C)} be met, $(g^-,g^+)\in\mathcal{G}$, and $V : \mathbb{R} \to \mathbb{C}$, $\phi : \mathbb{R} \to \mathbb{R}$ be continuous functions. Then $Q_n\to Q$ as $n \to\infty$.
\end{theorem}
\begin{proof}
We will use the method of weak convergence. First, we normalise the pseudo-transition probabilities $p_{n,k}$ and prove that the corresponding sequence of Markov chains converges weakly to the target diffusion process $X$ using convergence results from \cite{Liang2021}. Then we extend this convergence result to a sequence of auxiliary bivariate processes, where from the desired convergence $Q_n \to Q$ will follow as these quantities can be expressed as expectations of a suitable functional.

Let $\{\xi_{n,k} \}_{k=1}^n$ denote a Markov chain with transition probabilities $p_{n,k}(x,y)/C_{n,k}(x)$, $(x,y) \in E_{n,k-1}\times E_{n,k}$, where $C_{n,k}(x) := \sum_{y \in E_{n,k}}p_{n,k}(x,y)$. Furthermore, let $\Xbb$ denote the Brownian bridge-interpolated version of $\{\xi_{n,k}\}_{k=1}^n$, i.e.
\begin{equation*}
	\Xbb(t) := B_{n,k}^{\xi_{n,k-1},\xi_{n,k}}(t),\quad t\in [t_{n,k-1},t_{n,k}],\quad k=1,\ldots,n,
\end{equation*}
where
\begin{equation*}
	B_{n,k}^{x,y}(t) :=B_{n,k}^{\circ}(t) + x + n(t- t_{n,k-1})(y-x), \quad x ,y\in \mathbb{R},
\end{equation*}
and $B_{n,k}^{\circ}(t)$, $t \in [t_{n,k-1},t_{n,k}]$, are independent Brownian motions “pinned” at the time-space points $(t_{n,k-1}, 0)$ and $(t_{n,k}, 0)$, these bridges being independent of our chain. From Corollary 1 in \cite{Liang2021}, we know that $\Xbb \dto X$ as $n\to\infty$, where ``$\dto$'' denotes weak convergence of random elements in the respective functional space (in this case, in space $(C, \| \cdot \|_{\infty} )$).

Define the two-dimensional process $\mathbf{X}(t) :=(X(t),Z(t))$ by letting
\begin{equation*}
	Z(t) := \int_0^t V(X(s))\,ds, \quad t\in[0,1].
\end{equation*}
Then we can rewrite $Q$ from \eqref{defn:P} as
\begin{equation*}
	Q = \mathbf{E}\big(\Psi(\mathbf{X}(1));X \in G\big),
\end{equation*}
where $\Psi(\mathbf{X}(1)) := \phi(X(1))e^{Z(1)}$. Our approximation \eqref{defn:P_n} can also be rewritten as
\begin{equation*}
	Q_n = \mathbf{E}\bigl(\Psi(\mathbf{X}_n(1));\Xbb \in G\bigr),
\end{equation*}
where $\mathbf{X}_n(t) := (\Xbb(t), Z_{n, \lfloor n t\rfloor })$, $Z_{n,0}=0$ and, for $k=0,1,\ldots, n,$
\begin{equation*}
Z_{n,k} :=  \frac{\Delta_n}{2}\sum_{i=1}^{k} \bigl[V(\Xbb(t_{n,i-1})) +V(\Xbb(t_{n,i}))\bigr] .
\end{equation*}
We will now show that $(\Xbb, Z_{n}) \dto (X,Z)$ in the respective Skorokhod space by verifying the consistency of infinitesimal moments and tightness. We first define the auxiliary process $\widehat{Z}_n(t) := \widehat{Z}_{n,\lfloor n t\rfloor}$, $t\in[0,1]$, with $\widehat{Z}_{n,k} = \sum_{i=0}^{k-1} V(\Xbb(t_{n,i})) \Delta_n$ and show that $(\Xbb,\widehat{Z}_n)\dto (X,Z)$.

Define the joint process $\Xcc(t) := (t,\Xbb(t),\widehat{Z}_{n,\lfloor n t \rfloor })$ and the stopping time
\begin{equation*}
	\tau_n^r:= \inf\bigl\{t \geq 0: \| \Xcc(t)\| \vee \| \Xcc(t-)\| >r\bigr\},
\end{equation*}
$\| \mathbf{u}\| := \lvert u_1 \rvert\vee \lvert u_2\rvert \vee \lvert u_3\rvert$ being the Chebyshev norm of $\mathbf{u}=(u_1,u_2,u_3)\in\mathbb{R}^3$. Note that
\begin{equation*}
	\widehat{Z}_{n,k+1} = \widehat{Z}_{n,k} +  V(\Xbb(t_{n,k}))\Delta_n,\quad \widehat{Z}_{n,0} =0.
\end{equation*}
Let $\Delta\widehat{Z}_{n,k+1} := \widehat{Z}_{n,k+1} -\widehat{Z}_{n,k} $. It follows immediately that, for all $r>0$,
\begin{equation*}
	\lim_{n \to \infty}\max_{1\leq k\leq \lfloor n \tau_n^r\rfloor} \Bigl\lvert \Delta_n^{-1}\mathbf{E}\bigl(\Delta \widehat{Z}_{n,k+1} \,|\,\widehat{Z}_{n,k} = z, \Xbb(t_{n,k})=x\bigr) -V(x)\Bigr\rvert =0
\end{equation*}
and
\begin{equation*}
	\lim_{n \to \infty}\max_{1\leq k\leq \lfloor n \tau_n^r\rfloor}\Delta_n^{-1}\text{Var}\bigl(\Delta \widehat{Z}_{n,k+1} \,|\,\widehat{Z}_{n,k} = z, \Xbb(t_{n,k})=x\bigr) =0.
\end{equation*}
Tightness of the sequence of the distributions of $\Xcc$ in the respective Skorokhod space is immediate since $V$ is bounded on compact sets and hence it follows from Corollary~4.2 in \cite{EthierKurtz} that $(\Xbb,\widehat{Z}_n)\dto (X,Z)$. For all $r>0,$
\begin{equation*}
	\max_{1\leq k\leq \lfloor n \tau_n^r\rfloor}\lvert  \widehat{Z}_{n,k} - Z_{n,k}\rvert \leq \frac{\Delta_n}{2}\lvert V(0) + V(r)\rvert \xrightarrow{\text{a.s.}} 0, \quad n\to\infty,
\end{equation*}
and hence it follows that $(\Xbb,Z_n) \dto (X,Z)$ as well. By the continuous mapping theorem applied to the function $(x,y)\mapsto \phi(x)e^y$, the Portmanteau theorem (see e.g.~p.~24 in~\cite{Billingsley1968}), and Corollary 2 from~\cite{Liang2021}, the claimed result follows from the established weak convergence of the processes and boundedness of the integrands in $Q$. \hfill$\Box$

\end{proof}

\section{Numerical examples}\label{section:numerical}

In this section, we will illustrate the efficiency of our approximation scheme \eqref{defn:P_n}. The run times for these computations using the programming language \texttt{Julia} run on a MacBook Pro 2020 laptop computer with an i5 processor (2~GHz, 16~GB RAM) are basically the same as the ones reported in \cite{Liang2021} for computing boundary-crossing probabilities employing the same hardware and software. To evaluate the partial derivatives in \eqref{eqn:diffusion_coef}, we used the package \texttt{HyperDualNumbers.jl}

We apply our algorithm in three different examples, where $V(x)$ is equal to $-ix^2$, $x$, and $\kappa \mathbf{1}\{x > r\}$ respectively, for some $x,\kappa,r\in \mathbb{R}$. Due to the discontinuity of the function $V$ at $x=r$ in the last case, we modify the scheme slightly for that choice of $V$.

For each example, we compute $Q_n$ and plot the observed convergence rate of $Q_{n+1} - Q_n$ to zero as $n$ grows. Closed-form expressions for $Q$ are not available in any of these examples. Due to the iterative nature of our scheme \eqref{defn:P_n}, as a byproduct, we obtain approximations to $v(t,x) := (T_{t,1}\phi)(x)$, where $T_{s,t}$ appeared in \eqref{T}. We will plot the surfaces $v(t,x)$ for each example as well.

\subsection{The case $V(x) = -i x^2$}

To demonstrate general methods for obtaining distributional properties of Wiener functionals of the form $\int_0^1 V(W(s))\,ds$, the characteristic function of the random variable $\int_0^1 W(s)^2\,ds$ was computed in~\cite{Kac1949} and \cite{CameronMartin1945} yielding
\begin{equation*}
	\mathbf{E}\exp\biggl\{i \lambda \int_0^1W(u)^2\,du \biggr\} = \sqrt{\text{sech}\sqrt{2i\lambda}}, \quad \lambda \in\mathbb{R}.
\end{equation*}
Using our scheme, we compute approximations to the expressions of this form restricted to the event of non-crossing given boundaries:
\begin{equation*}\label{defn:v(t,x)_V1}
	v(t,x) = \mathbf{E}\biggl(\exp \biggl\{i  \int_t^1 W(u)^2\,du \biggr\}; g^-(s) < W(s) < g^+(s),\,s\in[t,1] \,\bigg|\, W(t)=x\biggr),
\end{equation*}
with time-dependent boundaries
\begin{equation*}
	g^{\pm}(t) = \pm 4 \mp t^2,\quad t\in [0,1].
\end{equation*}
In Fig~\ref{fig:complex_potential}, we plot our approximation of the form \eqref{defn:P_n} to $v(t,x)$. We see from Fig~\ref{fig:complex_potential_convergence}, that $\lvert \Re Q_{n+1} - \Re Q_n\rvert $ converges to zero at the rate $O(n^{-3})$, suggesting that $\lvert \Re Q_n - \Re Q \rvert=O(n^{-2})$ as $n\to\infty$.

\begin{figure}[ht]
	\centering
	\includegraphics[scale=0.55]{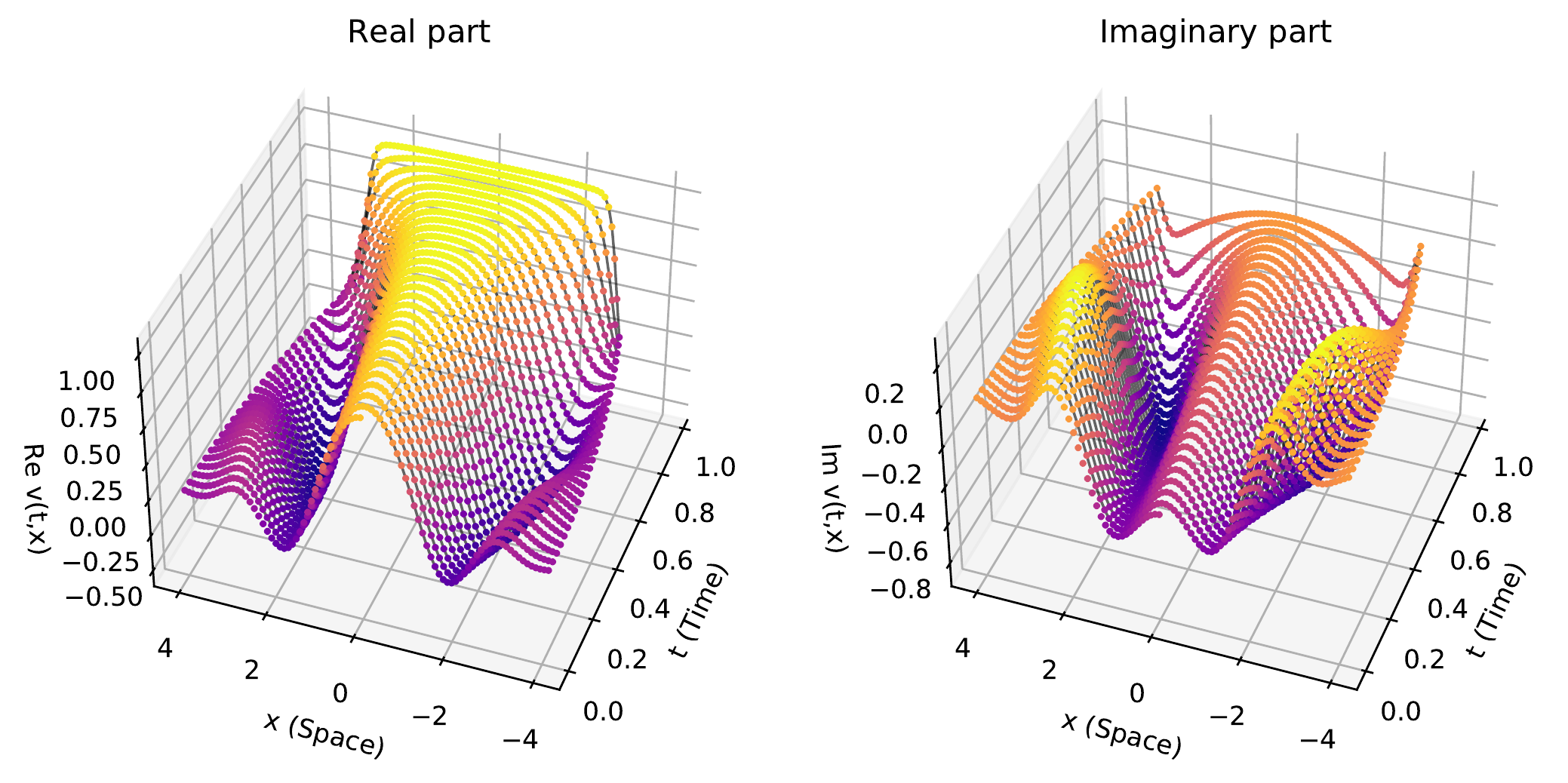}
	\caption{Approximations to $\Re v(t,x)$ and $\Im v(t,x)$ using \eqref{defn:P_n} with parameters: $n=30$, $\gamma =2$, $\delta =0$, $\mu(t,x) =0$, and $V(x) = -ix^2$.}
	\label{fig:complex_potential}
\end{figure}

\begin{figure}[ht]
	\centering
	\includegraphics[scale=0.7]{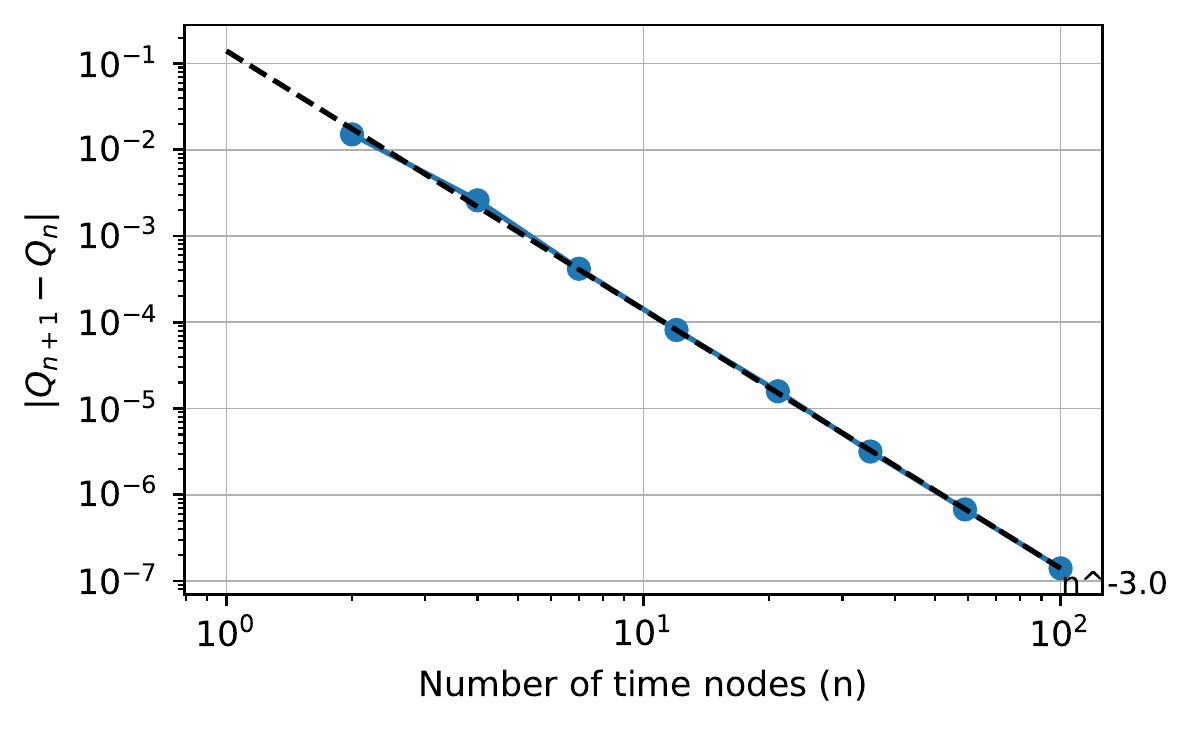}
	\caption{The log-log plot of $\lvert \Re Q_{n+1} - \Re Q_n \rvert$ is displayed by the blue line. The dashed straight line has a slope equal to $-3$, suggesting that $\lvert \Re Q_n - \Re Q\rvert $ is converging at rate $O(n^{-2})$. }
	\label{fig:complex_potential_convergence}
\end{figure}

\subsection{The case $V(x) = x$}

In this example, we price a zero-coupon bond with a two-sided time-dependent knock-out barrier, using a one-factor Hull--White model for spot interest rates. This problem was also explored in \cite{Kuan2003}. For $\alpha,\sigma >0,$ the Hull--White model assumes that the spot interest rate $r$ has the following dynamics:
\begin{equation*}
	dr(t) = (\theta(t) -  \alpha X(t))\,dt + \sigma \,dW(t), \quad t\geq 0,
\end{equation*}
where $\theta(t)$ depends on $\alpha$, $\sigma$ and the currently observed instantaneous forward rate curve $f(t)$ (see e.g.~Proposition 10.1.6 in \cite{Piterbarg2010}):
\begin{equation*}
	\theta(t) = f'(t) + \alpha f(t) + \frac{\sigma^2}{2\alpha}(1-  e^{-2\alpha t}),
\end{equation*}
assuming that $r(0) = f(0)$. To apply our scheme to this example, we scale all the space variables by $\sigma$ to transform $r(t)$ into a unit-diffusion process. Using the martingale pricing theorem, the fair price of a one-year maturity zero-coupon bond  with a barrier option feature is equal to
\begin{equation*}\label{defn:ZCB_price}
\mathbf{E} \bigl(e^{-\int_0^1 r(s)\,ds} ; r \in G\bigr).
\end{equation*}
We used the time-dependent boundaries given by
\begin{equation*}
	g^{\pm}(t) = \pm 0.04(1\mp \mbox{$\frac{1}{2}$}\sin(3t)),\quad t\in [0,1].
\end{equation*}
To demonstrate the performance of our scheme, we set $\alpha = 0.01$, $\sigma = 0.01$ and a flat instantaneous forward curve $f(t) \equiv 0.03$. In the left pane in Fig~\ref{ex2}, we plot our approximation of
\begin{equation}\label{defn:v(t,x)_V2}
	v(t,x) = \mathbf{E}\biggl(\exp \biggl\{- \int_t^1 r(u)\,du \biggr\}; g^-(s) < r(s) < g^+(s),\,s\in[t,1] \,\bigg|\, r(t)=x\biggr),
\end{equation}
for different initial values using our approximation algorithm \eqref{defn:P_n}. We see from the right pane in Fig~\ref{ex2} that $Q_{n+1} -  Q_n $ converges to zero at the rate $O(n^{-3})$, suggesting that $\lvert  Q_n - Q \rvert =O(n^{-2})$ as $n\to\infty$.

\begin{figure}[ht]
\includegraphics[scale=0.5]{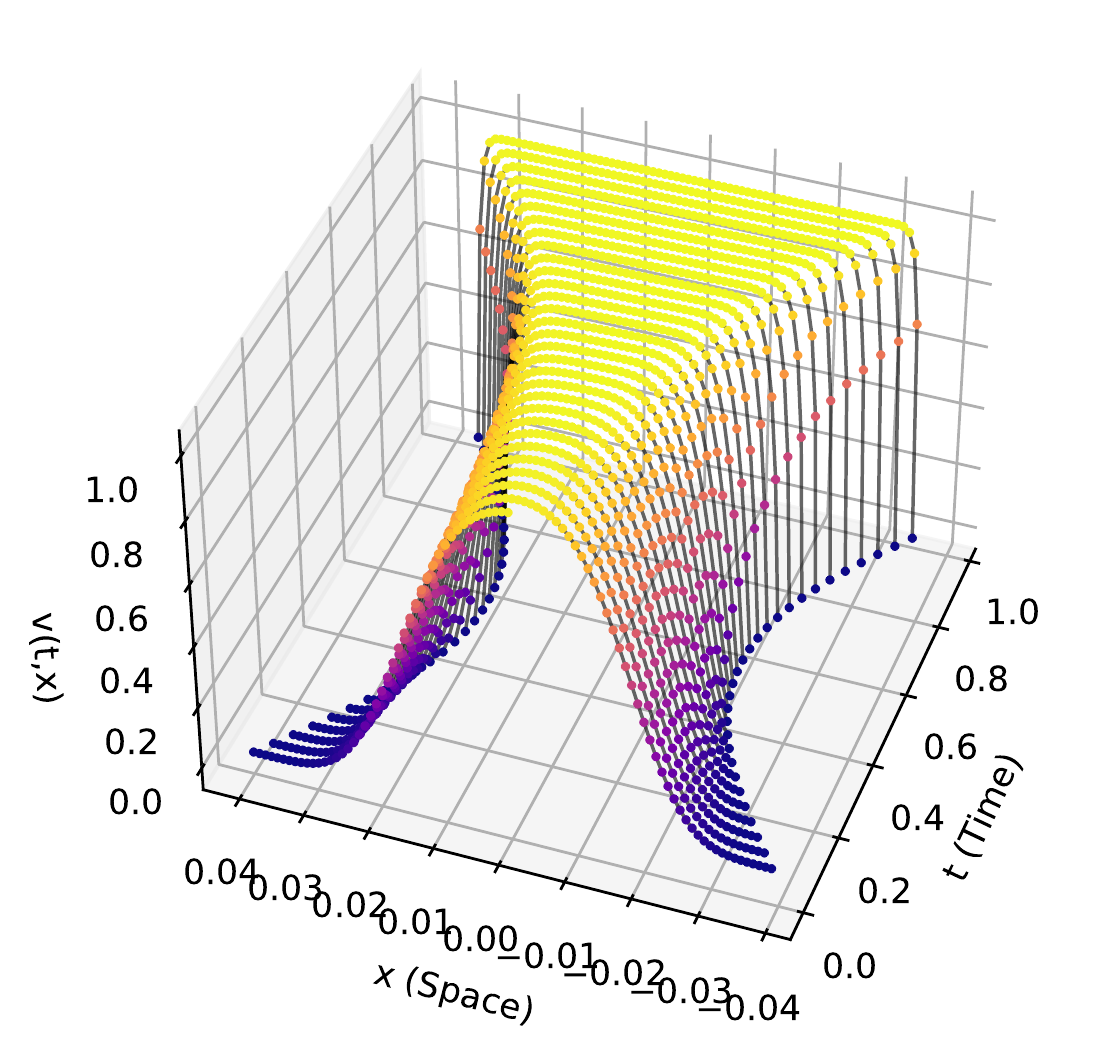}\quad
\includegraphics[scale=0.5]{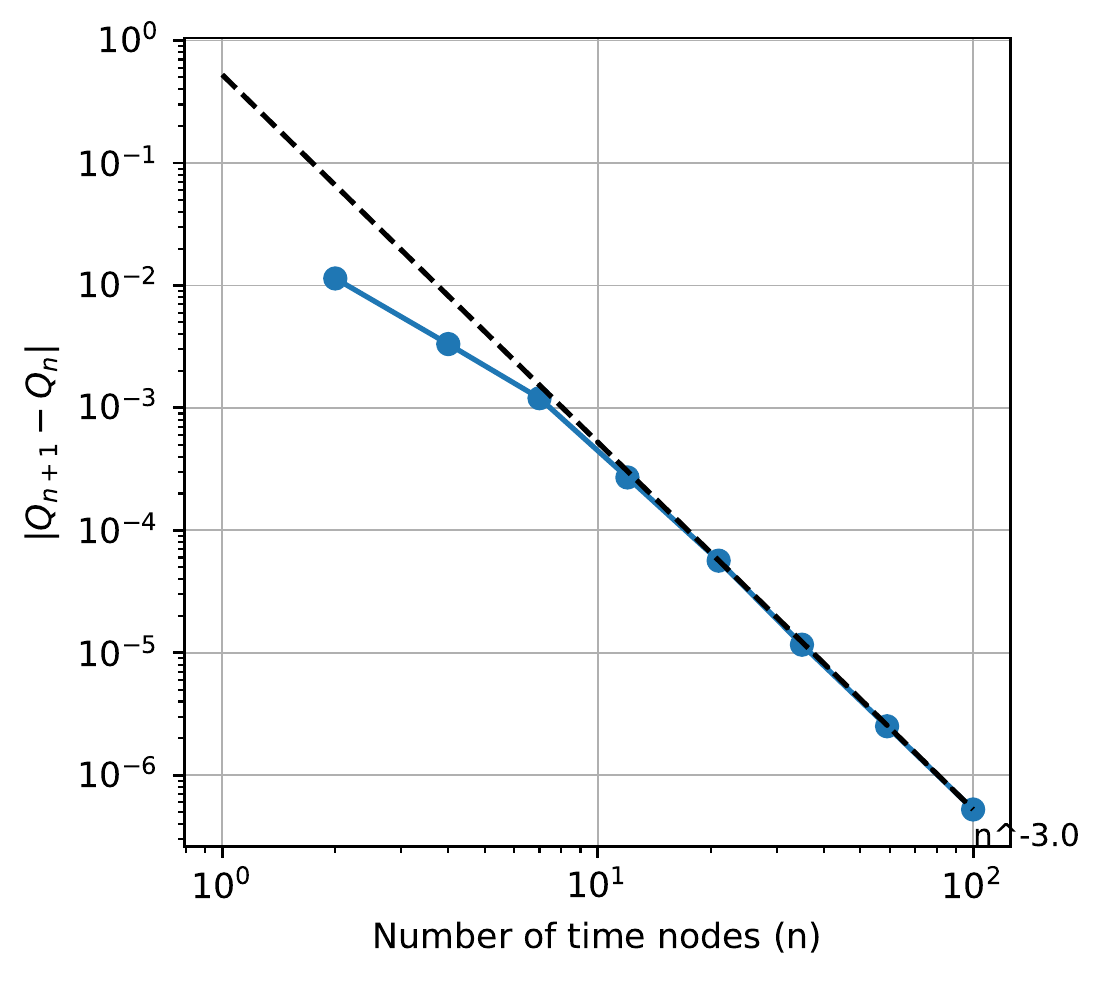}
\caption{The left pane shows an approximation to $v(t,x)$ from~\eqref{defn:v(t,x)_V2} for different initial values $(t,x)$ using \eqref{defn:P_n} with parameters $n=30$, $\gamma =2$, $\delta =0$, $\mu(t,x) =(\theta(t) -\alpha x)$, $\sigma  = 0.01$ and $V(x)=x$. On the right pane, we present a log-log plot of $\lvert Q_{n+1} - Q_n \rvert$ (the blue line). The dashed black straight line has a slope equal to $-3$.}
\label{ex2}
\end{figure}

\subsection{The case $V(x) = \kappa \mathbf{1}\{ x >r \}$}

The standard barrier option is one of the most popular first-generation exotic derivatives, due to its cheaper price compared to their vanilla counterparts. If an investor believes that a certain price is unlikely to fall below a certain level, they can choose to include a knock-out feature in the option, which reduces the price. These cost reductions can be substantial when volatility is high. However, standard barrier options come with a few disadvantages. Option buyers can lose their entire investment due to short time price spikes once the price is near the barrier. The delta sensitivity of such a barrier option is discontinuous near the boundary, making it difficult to hedge for options dealers.

The so-called ``step option'' was introduced in \cite{Linetsky1999} to address these issues. The payoff of an ``up-and-out'' step option with finite knock-out rate $\kappa \geq 0$ written on an option with payoff $\phi$ is given by the formula
\begin{equation*}
    \exp\{-\kappa \tau_r(T)\}\phi(X(T)),
\end{equation*}
where $\kappa>0$ is a constant and  $\tau_r(t)$ is the sojourn time of the underlying price process~$X$ above the level $r >0$ until time $t$:
\begin{equation*}
    \tau_r(t) := \int_0^t \mathbf{1}\{X(s)> r\}\,ds,\quad t\in [0,T].
\end{equation*}
We will assume $X$ to be the standard Brownian motion process $W$ in this example. 

This option can be modified by adding barrier features that will further reduce its price. We    introduce a time-dependent knock-out barrier $g^+$ above the step option barrier $r$ and a lower knock-out barrier $g^-$. We call this option a ``hybrid step-barrier option''. Its  payoff is   given by
\begin{equation*}
    \exp\{-\kappa \tau_r(T)\}\phi(X(T))\mathbf{1}\{ g^-(s) < X(s) < g^+(s),\,s\in [0,T]\}.
\end{equation*}
The expectation of this functional is a special case of our $Q$ from \eqref{defn:P} with $V(x) = \kappa \mathbf{1}\{x > r\}$. In this case, the trapezoidal approximation \eqref{defn:potential_correction} converges very slowly due to the discontinuity of~$V$ at the occupation level~$r$. We now describe a modification of $e_{n,k}$ to address this issue. Let $\mathbf{E}^{t,x,y}(\cdot) := \mathbf{E}(\,\cdot\,|\, W(0)=x, W(t)= y)$. Conditioning $W$ on endpoints $(0,x)$ and $(t,y)$, one can expect that for small times $t$, the sojourn time $\tau_r(t)$ is ``almost independent'' of the boundary crossing event:
\begin{align*}
	&\mathbf{E}^{t,x,y} \bigl[e^{-\kappa \tau_r(t)}\phi(X(t)); g^-(t)< X(s) <g^+(s),\,s\in[0,t]\bigr] \\
	&\qquad \approx \mathbf{E}^{t,x,y}e^{-\kappa \tau_r(t)}\mathbf{E}^{t,x,y}\bigl[\phi(X(t)); g^-(t) < X(s) < g^+(s),\,s\in [0,t]\bigr].
\end{align*}
There is a semi-explicit formula for $\mathbf{E}^{t,x,y}e^{-\kappa \tau_r(t)}$ given in (1.4.7) in \cite{Borodin2002}. However, efficient numerical computations of the double convolution present in that formula was challenging in the case when $x<r$ and $y>r$. Hence we decided to use the following simple approximation.

\begin{lemma}
For $x,y,\kappa \in \mathbb{R}$, as $t \downarrow 0$,
\begin{equation*}
	\mathbf{E}^{t,x,y}e^{-\kappa \tau_r(t)} = 1 - \kappa  \int_0^t \overline{\Phi}\left(\frac{r - x - \frac{s}{t}(y-x) }{\sqrt{(t-s)s/t} }\right)\,ds + O(t^2),
\end{equation*}
where $\Phi(\cdot)$ is the standard normal distribution function and $\overline{\Phi}(x):= 1 - \Phi(x)$.
\end{lemma}
\begin{proof}
Since $ \tau_r(t) \leq t$, a Taylor series expansion of $e^{-x}$ yields
\begin{equation*}
	\mathbf{E}^{t,x,y}e^{-\kappa \tau_r(t)} = 1 - \kappa \mathbf{E}^{t,x,y}\tau_r(t) + O(t^2).
\end{equation*}
Changing the order of integration using Fubini's theorem we get
\begin{equation*}
	\mathbf{E}^{t,x,y}\tau_r(t) = \int_0^t \mathbf{P}^{t,x,y}(W(s) >r)\,ds.
\end{equation*}
Since
\begin{equation*}
	(W(s) \,|\, W(0)=x,W(t)=y) \sim N\left( x + \frac{s}{t}(y-x), \frac{s}{t}(t-s) \right),
\end{equation*}
one has
\begin{equation*}
	\mathbf{E}^{t,x,y}\tau_r(t) = \int_0^t \overline{\Phi}\left(\frac{r - x - (y-x)s/t }{\sqrt{(t-s)s/t}}\right)\,ds.
\end{equation*}
\hfill$\Box$
\end{proof}

\begin{remark}
One could actually compute  the second-order term in the expansion for $\mathbf{E}^{t,x,y}e^{-\kappa \tau_r(t)}$ using the following formula for the second moment of~$\tau_r(t)$:
\begin{equation*}
	\mathbf{E}^{t,x,y}\tau_r(t)^2=\int_0^t\int_0^t\mathbf{P}^{t,x,y}(W(u) >r, W(s) >r)\,du\,ds.
\end{equation*}
Let $W^{t,x,y} := \{W^{t,x,y}(s) : s\in [0,t]\}$ denote the Brownian motion pinned at $(0,x)$ and $(t,y)$. For brevity, we also set $W^{t}:=W^{t,0,0}$. It is well-known that, for $s<t$,
\begin{equation*}
	\mathbf{E}W^{t,x,y}(s)  = x + \frac{s}{t}(y-x), \quad \var\, W^{t,x,y}(s) = \frac{s}{t}(t-s),
\end{equation*}
and hence
\begin{equation*}
	\mathbf{P}^{t,x,y}(W(u) >r,W(s) >r) = \mathbf{P}\left(\frac{W^{t}(u)}{\sqrt{\var\, W^t(u) }}>r_u,\frac{W^{t}(s)}{\sqrt{\var\, W^t(s)}} >r_s\right),
\end{equation*}
where
\begin{equation*}
	r_v:= \frac{r - x - (y-x)v/t}{\sqrt{(t-v)v/t}},\quad v\in[0,t].
\end{equation*}
Now since $\mathbf{E}W^t(s)W^t(u)= (t-u)s/t$, $s<u$, we get
\begin{align}
	\rho :& 
= \text{Cov}\left(\frac{W^t(s)}{\sqrt{\var\, W^t(s)}},\frac{W^t(u)}{\sqrt{\var\, W^t(u)}}\right) 
\notag \\
& = \frac{\mathbf{E}W^t(s)W^t(u)}{\sqrt{\var\, W^t(s) \var\, W^t(u)}}= \sqrt{\frac{s(t-u)}{u(t-s)}}.
\label{rho}
\end{align}
Therefore
\begin{equation*}
	\mathbf{P}^{t,x,y}(W(u) >r,W(s) >r) =  1-\Phi_{\rho}(-r_s,-r_u),
\end{equation*}
where $\Phi_{\rho}$ denotes the bivariate standard normal cumulative distribution function with correlation $\rho$ given by~\eqref{rho}. Hence we have
\begin{equation*}
	\mathbf{E}^{t,x,y}\tau_r(t)^2=\int_0^t\int_0^t (1-\Phi_{\rho}(-r_s,-r_u))\,du\,ds.
\end{equation*}
However, numerically computing this integral turned out to be too computationally expensive and did not improve the efficiency of the algorithm.
\end{remark}

\begin{figure}[ht]
	\includegraphics[scale=0.5]{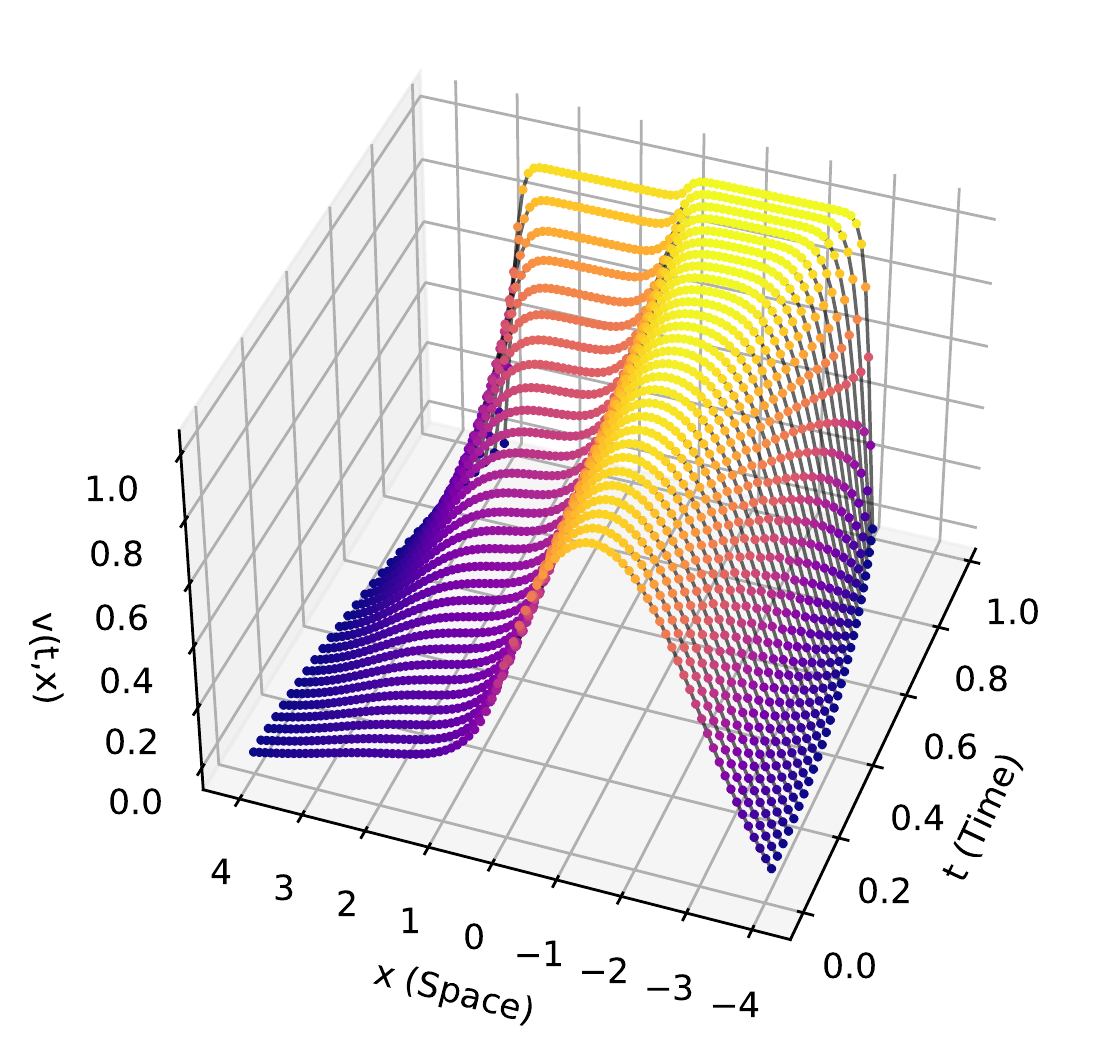}\quad
	\includegraphics[scale=0.5]{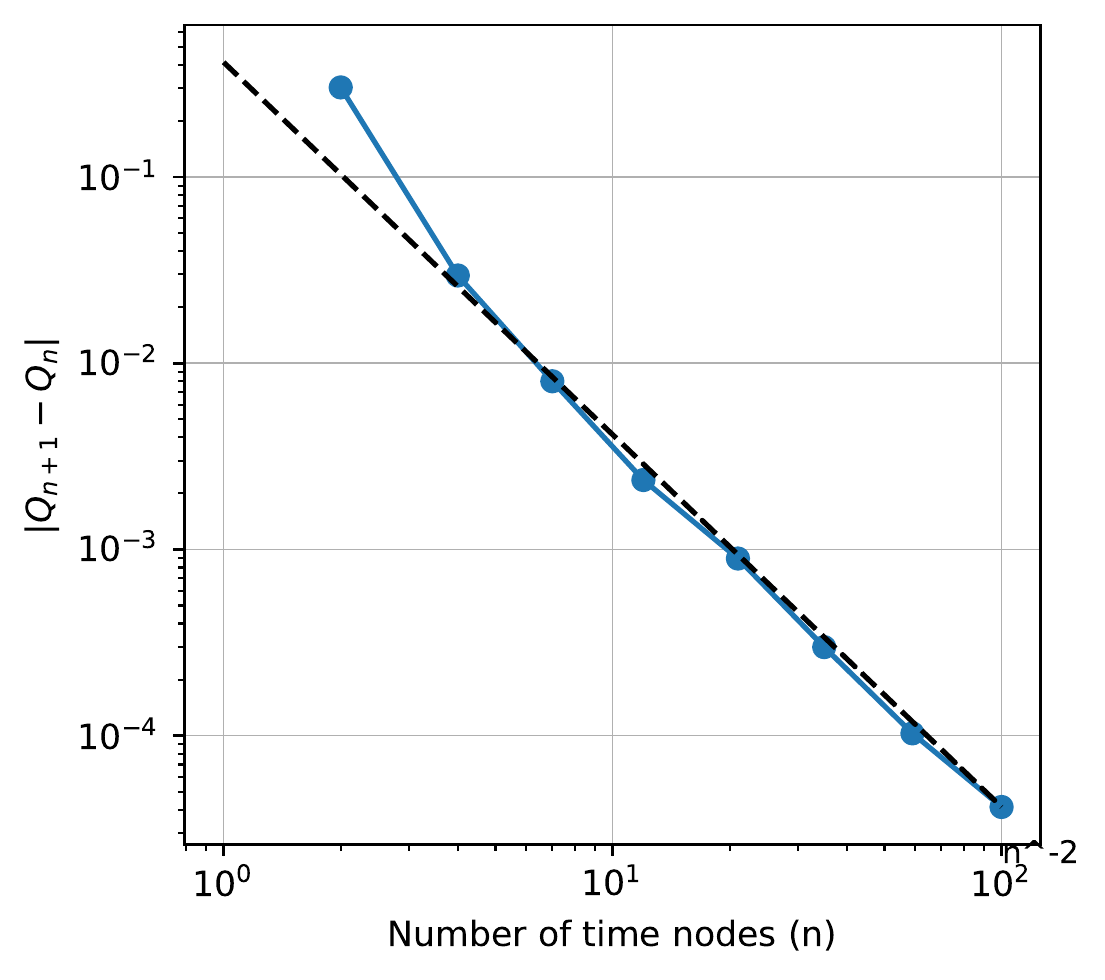}
	\caption{The left pane shows our approximation to $v(t,x)$ from \eqref{defn:v(t,x)_V3} using \eqref{defn:P_n} with modification \eqref{defn:potential_correction_2}. In this example we used parameters $n=30$, $r=1/19$, $\kappa =2$, $\delta =0$, $\mu(t,x) =0$, and $V(x)=\kappa\mathbf{1}\{x >r\}$. The right pane shows a log-log plot of $\lvert Q_{n+1} - Q_n \rvert$, displayed by the blue line. The dashed black straight line has a slope equal to $-2$, which indicates that $Q_n -Q = O(n^{-1})$.}
	\label{ex3}
\end{figure}

To summarise, our modification of the term $e_{n,k}$ in this special case when $V$ is discontinuous amounts to replacing \eqref{defn:potential_correction} with
\begin{equation}\label{defn:potential_correction_2}
	\tilde{e}_{n,k}(x,y) = 1 - \kappa \int_0^t \overline{\Phi}\left(\frac{r - x - (y-x)s/t }{\sqrt{(t-s)s/t}}\right)\,ds.
\end{equation}
In our numerical implementation, we computed the time integral appearing in~\eqref{defn:potential_correction_2} using Gaussian quadratures. To demonstrate the performance of our scheme, we set $r = 1/19$, $\kappa~=~2$ and $g^{\pm}(t) = \pm 4 \mp t^2$. Without replacing $e_{n,k}(x,y)$ with~\eqref{defn:potential_correction_2}, the convergence of $\lvert Q_{n+1} -Q_n\rvert$ to zero is highly non-smooth and slow. With the proposed modification, the scheme appears to converge at rate $O(n^{-1})$, as shown in the right pane of Fig~\ref{ex3}. The left pane in Fig~\ref{ex3} shows our approximation for the values of
\begin{equation} \label{defn:v(t,x)_V3}
	v(t,x) = \mathbf{E}\biggl(e^{- \kappa\int_t^1 \mathbf{1}\{W(s) > r\}\,du }; g^-(s) < W(s) < g^+(s),\,s\in[t,1] \,\bigg|\, W(t)=x\biggr).
\end{equation}

\end{document}